\documentclass[11pt]{article}

\usepackage{amsmath}
\usepackage{amsthm}
\usepackage{amsfonts}
\usepackage{bbm}

\usepackage{xcolor}

\DeclareMathOperator{\1}{\mathbbm{1}}

\theoremstyle{plain}
\newtheorem{thm}{Theorem}
\newtheorem*{thm*}{Theorem}
\newtheorem{lem}{Lemma}
\theoremstyle{definition}
\newtheorem{defn}{Definition}

\theoremstyle{remark}
\newtheorem{remk}{Remark}

\newcommand{\be}{\begin{equation}}
\newcommand{\ee}{\end{equation}}

\newcommand{\lin}{\underset{n\to\infty}{\lim}}

\newcommand{\ve}{\varepsilon}
\newcommand{\vf}{\varphi}

\newcommand{\mbR}{{\mathbb R}}
\newcommand{\mbN}{{\mathbb N}}

\renewcommand{\ln}{\log}

\newcommand{\cF}{{\cal F}}

\newcommand{\supp}{\supp{\rm supp}}

\newcommand{\E}{\mathrm{E}}
\newcommand{\Pb}{\mathrm{P}}

\newcommand{\toP}{\stackrel{\mathrm{P}}{\to}}
\newcommand{\toas}{\stackrel{ {a.s.}}{\to}}

\newcommand{\mbZ}{\mathbb{Z}}

\begin{document}
\title{A functional limit theorem for for excited random walks}

\author{ 
Andrey Pilipenko\footnote{Institute of Mathematics, National
Academy of Sciences of Ukraine, Kyiv, Ukraine,  and National Technical University of Ukraine ``KPI'', Kyiv, Ukraine
\newline e-mail: pilipenko.ay@yandex.ua} }
\maketitle
\begin{abstract}
\noindent We consider the limit behavior of an  excited random walk (ERW), i.e., a random walk 
whose  transition probabilities  depend on the number of times the walk has visited to 
the current state. We prove that  an ERW being naturally scaled
converges in distribution to an excited Brownian motion that satisfies an SDE, where the drift of the unknown process depends 
on its local time. Similar result was obtained by Raimond and Schapira, their  proof was based on  the Ray-Knight type theorems.
We propose a new method of investigations based on a study of the Radon-Nikodym density of  the ERW  distribution with respect to the distribution of a
 symmetric random walk.
\end{abstract}

\section{Introduction and results}
Let $\{X(k), k\geq 0\}$ be a sequence of $\mbZ$-valued random variables such that $|X(k+1)-X(k)|=1, k\geq 0.$
Denote by $\cF_n:=\sigma(X(0),X(1),\dots,X(n))$
 the filtration generated by $\{X(k)\}$ .
\begin{defn}
A random walk (RW) $\{X(k)\}$ is called an excited random walk (ERW) associated with a (may be random) sequence $\{\ve_i, i\geq 0\}\subset (-1,1)$ if
\be\label{eq_ERW}
\Pb\left(X(k+1)-X(k)=1 | \cF_k   \right)= 1- \Pb\left(X(k+1)-X(k)=-1 | \cF_k   \right)= p_i,
\ee
where $i= |\{j\leq k \; : \, X(j)=X(k)\}|,\ p_i=\frac12 (1+\ve_i).$
\end{defn}
Note that $\{X(k)\}$ is not a Markov chain, generally, and the study of traditional topics of the theory of stochastic processes such as 
recurrence, invariance principles, etc., is a non-trivial one for ERW. It demands new ideas and approaches, see for example \cite{Basdevant, Dolgopyat, KosyginaZerner, Merkl, PemantleVolkov, RaimondSchapira_Excited, Zerner} and references therein.

It was proved by Raimond and Schapira \cite{RaimondSchapira_Excited} that if 
$\ve_i=\ve_i^{(n)}=\frac{1}{\sqrt{n}}\vf(\frac{i}{\sqrt n})$, where $\vf $ is a 
bounded Lipschitz function, then the sequence of processes
$\{X_n(t):= \frac{X^{(n)}([nt])}{\sqrt n} , \ t\geq 0\}_{n\geq 1}$ converges in distribution 
in $D([0,\infty))$ to excited Brownian motion that is a solution to the following SDE
$$
dY(t)=\vf(L_Y(t,{Y_t})) dt + dW(t), 
$$
where $W$ is a Wiener process, $L_Y(t,x)$ is the local time of $Y$ at $x.$

They studied  the process
$\nu(i,k)=|\{j\leq i\ : \ X(j)=k\}|$ as a function of the spatial coordinate $k.$  It was proved that some scaling of $\nu$
taken at some Markov moments converges   to a solution of a Bessel type SDE
 that appears in a spirit of the Ray-Knight theorem, see also \cite{Norris}. Then a sequence $X(k)$ (and a process $Y_t$) were reconstructed from $\nu$ (and the local time $L_Y$, respectively).
 The corresponding proofs used the neat martingale 
technique. However the number of details they have checked was really large.

We propose a different method for proving of the corresponding result. We study the Radon-Nikodym density of $\{X(k), 0\leq k\leq n\}$
with respect to the distribution of a symmetric RW. Then we use Gikhman and Skorokhod result \cite{GS_abs} on absolute continuity of the limit process together with the Skorokhod theorem on a single probability space, and invariance principle  for the local times of random walks \cite{Borodin}. 

This method was used in \cite{P_Khomenko}  for studying the limit behavior of an RW with modifications at 0 whose
 transition probabilities are defined
as in \eqref{eq_ERW}, 
 where 
$$
i= |\{j\leq k \; : \, X(j)=0\}|,\ p_i=(\frac12+i\Delta)\wedge 1,
$$
 $\Delta>0$   is a size of modifications. It was proved there that
   $ X_n\Rightarrow X_\infty $ in the scheme of series, 
   where $\Delta_n=c n^{-\alpha}, \ c>0, \ \alpha>0,$ 
   $$
   X_n(t)= 
\begin{cases}
\frac{X_{\Delta_n}([nt])}{\sqrt n}, & \alpha\geq 1,\\
 \frac{X_{\Delta_n}([nt])}{ n^{1-\frac{\alpha}2}}, & \alpha\in(0,1),
\end{cases}\ \ 
   X_\infty(t)= 
\begin{cases}
 W(t), & \alpha> 1,\\
 \sqrt{c}\int_0^t L_{X_\infty}(s,0) ds+W(t), & \alpha= 1,\\
 \eta t, & \alpha\in(0,1),
\end{cases}
$$
$\eta$ is a non-negative random variable with the distribution function
$$
\Pb(\eta\leq x)=1-e^{-\frac{x^2}2},\ \ x\geq 0.
$$

\section{Main Result and Proofs } 
Let $\{\omega_k\}$ be a stationary ergodic sequence.
Consider a sequence of ERWs $\{X^{(n)}(k), k\geq 0\}_{n\geq 1}$  such that for a fixed $\omega=\{\omega_k\}$ the quenched probability satisfies the condition
\be\label{eq_ERW1}
\Pb_\omega\left(X^{(n)}(k+1)-X^{(n)}(k)=1 | \cF^{(n)}_k   \right)= 1- \Pb_\omega\left(X^{(n)}(k+1)-X^{(n)}(k)=-1 | \cF^{(n)}_k   \right)= p^{(n)}_{i,k},
\ee
where 
$$\cF^{(n)}_k:=\sigma(X^{(n)}(0),X^{(n)}(1),\dots,X^{(n)}(k)), \
 i= |\{j\leq k \; : \, X^{(n)}(j)=X^{(n)}(k)\}|,
$$
$$
 p^{(n)}_{i,k}=\frac12 (1+\ve^{(n)}_{k,X^{(n)}(k),i}),\ \ve^{(n)}_{k,x,i}=n^{-1/2}\vf(\frac kn, \frac{x}{\sqrt{n}}, \frac{i}{\sqrt{n}}, \omega_k)
$$
Here $\vf$ is a fixed bounded measurable function.
 
 The annealed, or averaged, probability will be denoted by $\Pb.$
 
Set $X_n(t)=\frac{X^{(n)}([nt])}{\
\sqrt{n}},\ n\in\mbN,\ t\geq 0.$ For convenience we will assume  that $X_n(0)=0.$

Let $D([0,\infty))$ be the space of cadlag functions equipped with the Skorokhod $J_1$ topology, see \cite{Billingsley}. 
\begin{thm}\label{thm1}
Assume that the function $\vf:[0,\infty)\times\mbR\times [0,\infty)\times\mbR\to\mbR$ is bounded and uniformly continuous. Then the
 sequence $\{X_n(\cdot),\ n\geq 1\}$
converges in distribution in $D([0,\infty))$ with respect to almost every quenched measure $P_\omega$, and also with respect to the averaged measure $\Pb$, 
to a solution of the SDE
\be\label{eq_EBM}
Y_t=\int_0^t \bar \vf(s,Y_s,L_Y(s,Y_s))ds +W(t),\ t\geq 0,
\ee
where $\bar \vf(t,x,l)=\E \vf(t,x,l,\omega_k)$, $W$ is a Wiener process.
\end{thm}
\begin{remk}
There is a unique weak solution to \eqref{eq_EBM} by Girsanov's theorem, see \cite{Norris}.
\end{remk}
\begin{proof}
In order to explain the idea of the proof and to avoid cumbersome calculations, at first we prove the theorem for $\vf$ that depends only on the first three of its coordinates, i.e.,
$\vf(t,x,l,\omega)=\vf(t,x,l).$ Then we explain how to handle the general case.

Denote by $\{S(k), k\geq 0\}$ a symmetric random walk, $S(k)= \xi_1+\dots+\xi_k,\ S(0)
=0$, where $\{\xi_i\}$ are i.i.d., $\Pb(\xi_i=\pm1)=1/2.$

 Let  $P_{X^{(n)}}$ be a distribution of $\{X^{(n)}(k)\}_{k=0}^n,$ $P_{S^{(n)}}$ be a distribution of $\{S(k)\}_{k=0}^n,$

Then $P_{X^{(n)}}\ll P_{S^{(n)}} $ and the Radon-Nikodym density equals:
\begin{multline}
\forall  i_0=0,i_1,...,i_n\in\mbZ,\ |i_{k+1}-i_k|,\\
\frac{d P_{X^{(n)}}}{d  P_{S^{(n)}}}(i_0,i_1,...,i_n)=
 \underset{k=0}{\overset{n-1}\Pi}\frac{\frac12 (1+\ve^{(n)}_{k})}{\frac 12}= \underset{k=0}{\overset{n-1}\Pi}{(1+\ve^{(n)}_{k})} =
\end{multline}
$$
\underset{k=0}{\overset{n-1}\Pi} \left(1+\frac 1{\sqrt n} \vf(\frac kn, \frac{i_k}{\sqrt{n}}, \frac{l(k, i_k )}{\sqrt{n}})
\1_{i_{k+1}-i_k=1}- \frac 1{\sqrt n} \vf(\frac kn, \frac{i_k}{\sqrt{n}}, \frac{l(k, i_k )}{\sqrt{n}})
 \1_{i_{k+1}-i_k=-1}\right)=
$$
$$
\underset{k=0}{\overset{n-1}\Pi} \left(1+\frac 1{\sqrt n} \vf(\frac kn, \frac{i_k}{\sqrt{n}}, \frac{l(k, i_k )}{\sqrt{n}})
(i_{k+1}-i_k)\right),
$$
where $l(k, i)= |\{ j\leq k\, : \, X^{(n)}(j)=i\} |$

Hence
$$
\frac{d P_{X^{(n)}}}{d  P_{S^{(n)}}}(S(0), S(1), ...,S(n))=
$$ 
\be\label{eq224}
\underset{k=0}{\overset{n-1}\Pi} \left(1+\frac 1{\sqrt n} \vf(\frac kn, \frac{S(k)}{\sqrt{n}}, \frac{\nu (k, S(k) )}{\sqrt{n}})
\xi_{k+1} \right),
\ee
where $\nu (k, i)= |\{ j\leq k\, : \, S(j)=i\} |$.

\begin{lem}\label{lemGS}
Let $\{X^n, n\geq 1\}$ and  $\{Y^n, n\geq 1\}$ be sequences of random elements given on the same probability space and taking values
in a complete separable metric space $E$.

Assume that 

1)   $Y_n\overset{{\rm P}}{\to} Y_0, n\to \infty$;

2) for each $n\geq 1$ we have the absolute continuity of the distributions
$$
\Pb_{X_n}\ll \Pb_{Y_n};
$$
3) the sequence $\{\rho_n(Y_n), n\geq 1\}$ converges in probability to a random variable $p,$ 
 where $\rho_n=\frac{d\Pb_{X_n}}{d\Pb_{Y_n}}$ is the Radon-Nikodym density;

4)
 $\E p=1.$
 
Then the sequence of distributions $\{P_{X_n}\}$  converges weakly as $n\to\infty$ to the probability measure
 $\E(p\, |\, Y_0=y) P_{Y_0} (dy)$.
\end{lem}
   The idea of the proof of the lemma is due to Gikhman and Skorokhod \cite{GS_abs}. Since   $\{\rho_n(Y_n), n\geq 1\}$  are non-negative random variables $\E\rho_n(Y_n)=1,$  the condition   $\E p=1$ yields the uniform integrability of $\{\rho_n(Y_n), n\geq 1\}.$ 
The proof of Lemma \ref{lemGS} follows from the next calculations
$$
\forall f\in C_b(E):\ \ 
 \lin \int_E f dP_{X_n}=\lin \E f(X_n)= \lin \E f(Y_n) \rho_n(Y_n)=
   \E f(Y_0) p = 
$$
\be\label{eq266}
\E\left( f(Y_0) \, \E(p\, | \, Y_0)\right)=
\int_E f(y) \E(p \,|\, Y_0=y) P_{Y_0} (dy).
\ee

Let us continue the proof of  Theorem \ref{thm1}. We will prove convergence in distribution $\frac{X^{(n)}([n\cdot])}{\sqrt n}\Rightarrow Y$ in   $D([0,1])$ only.

 We need the following invariance principle for RWs and the local times of RWs.
 \begin{thm}\label{thm_Borodin}
 There is  a probability space and  copies   $\{S^{(n)}(k), k=0,...,n\} \overset{d}= \{S(k), k=0,...,n\}$ defined on this space, and   a Wiener process
 $W(t), t\in[0,1],$   such that
 \be\label{eq6.0}
 \lim_{n\to\infty}\sup_{t\in[0,1]}|\frac{S^{(n)}([nt])}{\sqrt n}-W(t)|=0,
 \ee
\be\label{eq6.1}
 \lim_{n\to\infty}\sup_{t\in[0,1]}\sup_{x\in\mbR}|\frac{\nu^{(n)}([nt], [x\sqrt n])}{\sqrt n}-L_W(t,x)|=0,
 \ee 
 with probability 1, where  $\nu^{(n)}(k, i)= |\{ j\leq k\, : \, S^{(n)}(j)=i\} |$, $L_W$ is the local time of the Wiener process (we consider 
 a modification of $L_W$ that is continuous in $t,x$).
\end{thm}
Let us apply Lemma \ref{lemGS}, where  
$$
X_n=X_n(t)=\frac{X^{(n)}([nt ])}{\sqrt n},\ \   Y_n= S_n(t)=\frac{S^{(n)}([nt])}{\sqrt n},\, t\in[0,1].
$$
 It follows from \eqref{eq224} 
that 
$$
\ln\frac{d P_{X_n}}{d  P_{S_n}}(S_n)=
$$ 
$$
\underset{k=0}{\overset{n-1}\sum} \ln\left(1+\frac 1{\sqrt n} \vf(\frac kn, \frac{S^{(n)}(k)}{\sqrt{n}}, \frac{\nu^{(n)}(k, S^{(n)}(k) )}{\sqrt{n}})
\xi^{(n)}_{k+1} \right)=
$$
$$
\frac 1{\sqrt n}\underset{k=0}{\overset{n-1}\sum}   \vf(\frac kn, \frac{S^{(n)}(k)}{\sqrt{n}}, \frac{\nu^{(n)}(k, S^{(n)}(k) )}{\sqrt{n}})
\xi^{(n)}_{k+1} - \frac 1{2n} \underset{k=0}{\overset{n-1}\sum}  \vf^2(\frac kn, \frac{S^{(n)}(k)}{\sqrt{n}}, \frac{\nu^{(n)}(k, S^{(n)}(k) )}{\sqrt{n}})   +
$$
$$
\frac \theta{3n^{3/2}}\underset{k=0}{\overset{n-1}\sum}   |\vf^3(\frac kn, \frac{S(k)}{\sqrt{n}}, \frac{\nu^{(n)}(k, S^{(n)}(k) )}{\sqrt{n}})|=I_1^n+I_2^n+I_3^n,
$$
where $\theta\in(-1,1).$
Since $\vf$ is bounded, $\lim_{n\to\infty}I_3^n=0$ for all $\omega.$

By  \eqref{eq6.0}, \eqref{eq6.1},  continuity of $L_W(t,x)$ in both of its arguments, and dominated convergence theorem we have convergence
\be\label{eq7.1}
\lim_{n\to\infty}\frac 1{2n} \underset{k=0}{\overset{n-1}\sum}  \vf^2(\frac kn, \frac{S^{(n)}(k)}{\sqrt{n}}, \frac{\nu^{(n)}(k, S^{(n)}(k) )}{\sqrt{n}})= 
\ee
$$
\lim_{n\to\infty}\frac 1{2n} \int_0^1\vf^2(\frac{[nt]}{n}, \frac{S^{(n)}(\frac{[nt]}{n})}{\sqrt{n}}, \frac{\nu^{(n)}([nt], S^{(n)}([nt]) )}{\sqrt{n}}) dt=
$$
$$
  \frac 12 \int_0^1\vf^2( t, W(t), L_W(t, W(t) ) dt.
$$
\begin{lem}\label{lem2}
We have convergence in probability
$$
 \frac 1{\sqrt n}\underset{k=0}{\overset{n-1}\sum}   \vf(\frac kn, \frac{S^{(n)}(k)}{\sqrt{n}}, \frac{\nu^{(n)}(k, S^{(n)}(k) )}{\sqrt{n}})
\xi^{(n)}_{k+1} \toP \int_0^1 \vf ( t, W(t), L_W(t, W(t) ) dW(t), \ n\to\infty.
$$
\end{lem}
\begin{proof}
We use idea of Skorokhod \cite[Chapter 3, \S3]{SkorIssl}. 
Let $m\in\mbN$ be fixed. Then
$$
 |\frac 1{\sqrt n}\underset{k=0}{\overset{n-1}\sum}   \vf(\frac kn, \frac{S^{(n)}(k)}{\sqrt{n}}, \frac{\nu^{(n)}(k, S^{(n)}(k) )}{\sqrt{n}})
\xi^{(n)}_{k+1}- \int_0^1 \vf ( t, W(t), L_W(t, W(t) ) dW(t)|\leq
$$
\begin{multline*}
\left|\sum_{j=0}^{m-1}\sum_{[\frac{jn}{m}]\leq k<[\frac{(j+1)n}{m}]}  
 \left(\vf(\frac kn, \frac{S^{(n)}(k)}{\sqrt{n}}, \frac{\nu^{(n)}(k, S^{(n)}(k) )}{\sqrt{n}})
 - \right.\right.
\\
\left.\left.  \vf(\frac{[{jn}/{m}]}n, \frac{S^{(n)}([\frac{jn}{m}])}{\sqrt{n}}, \frac{\nu^{(n)}([\frac{jn}{m}], S^{(n)}([\frac{jn}{m}]) )}{\sqrt{n}})
  \right)\frac{\xi^n_{k+1}}{\sqrt{n}}\right|+
\end{multline*}
\begin{multline*}
\left|\sum_{j=0}^{m-1}  \left(\vf(\frac{[{jn}/{m}]}n, \frac{S^{(n)}([\frac{jn}{m}])}{\sqrt{n}}, \frac{\nu^{(n)}([\frac{jn}{m}], S^{(n)}([\frac{jn}{m}]) )}{\sqrt{n}})
  \right) \right.\\
\left.  \left( \Big(\sum_{[\frac{jn}{m}]\leq k<[\frac{(j+1)n}{m}]}\frac{\xi^n_{k+1}}{\sqrt{n}}\Big)
- \Big(W({\frac{[{(j+1)n}/{m}]}n})-W({\frac{[{jn}/{m}]}n})\Big)\right)\right|+  
\end{multline*}
\begin{multline*} 
 \left|\sum_{j=0}^{m-1} \Big(\vf(\frac{[{jn}/{m}]}n, \frac{S^{(n)}([\frac{jn}{m}])}{\sqrt{n}}, \frac{\nu^{(n)}([\frac{jn}{m}], S^{(n)}([\frac{jn}{m}]) 
)}{\sqrt{n}})-
\right.
\\
\vf ( \frac{[{jn}/{m}]}n, W(\frac{[{jn}/{m}]}n), L_W(\frac{[{jn}/{m}]}n, W(\frac{[{jn}/{m}]}n) )\Big)
\\
\left.  
\Big(W({\frac{[{(j+1)n}/{m}]}n})-W({\frac{[{jn}/{m}]}n})\Big)\right|+
\end{multline*}
\begin{multline*}
  \Big|\sum_{j=0}^{m-1} \int_{\frac{[{jn}/{m}]}n}^{\frac{[{(j+1)n}/{m}]}n}
 \Big(\vf ( \frac{[{jn}/{m}]}n, W(\frac{[{jn}/{m}]}n), L_W(\frac{[{jn}/{m}]}n, W(\frac{[{jn}/{m}]}n) ) -   
  \\
 \vf ( t, W(t), L_W(t, W(t) )   \Big)dW(t)\Big|
  =
\end{multline*}
$$
=I^{n,m}_1+I^{n,m}_2+I^{n,m}_3+ I^{n,m}_4.
$$
It follows from Theorem \ref{thm_Borodin}, Lebesgue dominated convergence theorem, and continuity of $L_W(t,x)$ in both of its arguments
that 
\be\label{eq378}
\lin \E(I_1^{n,m})^2= 
\ee
\begin{multline*}
\lin \frac1n\, \E  
 \sum_{j=0}^{m-1}\sum_{[\frac{jn}{m}]\leq k<[\frac{(j+1)n}{m}]}  
 \left(\vf(\frac kn, \frac{S^{(n)}(k)}{\sqrt{n}}, \frac{\nu^{(n)}(k, S^{(n)}(k) )}{\sqrt{n}})
\right. 
\\
 \left.
 - \vf(\frac{[{jn}/{m}]}n, \frac{S^{(n)}([\frac{jn}{m}])}{\sqrt{n}}, \frac{\nu^{(n)}([\frac{jn}{m}], S^{(n)}([\frac{jn}{m}]) )}{\sqrt{n}})
  \right)^2=
  \end{multline*}
$$
  \E  
 \sum_{j=0}^{m-1}\int_{\frac{j}{m}}^{\frac{(j+1)}{m} }  
 \left(\vf(t, W(t), L_W(t, W(t)  ) )
 - \vf(\frac{j}{m} ,W(\frac{j}{m}), L_W(\frac{j}{m},W(\frac{j}{m})) ) 
  \right)^2dt.
  $$
  $$
=\lin \E(I^{m,n}_4)^2
$$
It follows from Theorem \ref{thm_Borodin} that  $\lin  I^{n,m}_2 =\lin   I^{n,m}_3 =0$ a.s.  for each fixed $m$. So, by dominated convergence theorem
$$
\forall m\geq 1\ \ \lin \E(I^{n,m}_2)^2=\lin \E(I^{n,m}_3)^2=0.
$$ 
So for any $m\geq 1$
$$
\limsup_{n\to\infty}\E \left(\sum_{k=0}^{n-1}  
 \frac 1{\sqrt n}\underset{k=0}{\overset{n-1}\sum}   \vf(\frac kn, \frac{S^{(n)}(k)}{\sqrt{n}}, \frac{\nu^{(n)}(k, S^{(n)}(k) )}{\sqrt{n}})
\xi^{(n)}_{k+1}- \int_0^1 \vf ( t, W(t), L_W(t, W(t) ) dW(t)\right)^2\leq
$$
\be\label{eq411}
  4\E  
 \sum_{j=0}^{m-1}\int_{\frac{j}{m}}^{\frac{(j+1)}{m} }  
 \left(\vf(t, W(t), L_W(t, W(t)  ) )
 - \vf(\frac{j}{m} ,W(\frac{j}{m}), L_W(\frac{j}{m},W(\frac{j}{m})) ) 
  \right)^2dt.
  \ee
Letting $m\to\infty$ we    
 complete  the proof of the lemma.
 
\end{proof}
Since $\vf$ is bounded, 
\be\label{eq_int}
\E \exp\{ \int_0^1 \vf ( t, W(t), L_W(t, W(t) ) dW(t) - \frac 12 \int_0^1\vf^2( t, W(t), L_W(t, W(t) ) dt\}=1
\ee 
by Novikov's theorem.

Therefore, by Lemma \ref{lemGS} we have convergence $X_n\Rightarrow Y,$ where the distribution of $Y$ has a density
$\exp\{ \int_0^1 \vf ( t, W(t), L_W(t, W(t) ) dW(t) - \frac 12 \int_0^1\vf^2( t, W(t), L_W(t, W(t) ) dt\}$ with respect to the
 Wiener measure. Note that the local time and the integrals are measurable functions with respect to the $\sigma$-algebra 
 generated by $W.$ So there was no necessity  for calculations of the  conditional  expectation in Lemma \ref{lemGS}.
By Girsanov's theorem, the process $Y$ is a weak solution to the equation \eqref{eq_EBM}. The theorem is proved if $\vf(t,x,l,\omega)=\vf(t,x,l).$

Consider the general case.

We prove the theorem if we  find the corresponding limits in \eqref{eq7.1},
 \eqref{eq378}, and \eqref{eq411}, where the general summand is
replaced by 
$$
\vf(\frac kn, \frac{S^{(n)}(k)}{\sqrt{n}}, \frac{\nu^{(n)}(k, S^{(n)}(k) )}{\sqrt{n}}, \omega_k),
$$
and the sequence $\{\omega_k, k\geq 0\}$ is independent of $\{S^{(n)}(k) \}.$

 The next statement completes the proof of the theorem.
 \begin{lem}\label{lem3}
 Let $f:\mbR^{d+1}\to\mbR$  be a uniformly continuous and bounded  function, $\{\eta_k, k\geq 0\}_{n\geq 1}$ be a stationary ergodic sequence, $\{\xi_n(t), \ t\geq 0\}_{n\geq 1}$ be a sequence of continuous $\mbR^d$-valued processes that   locally uniformly converge to a process $\xi(t), t\geq0,$  almost surely,
 $$
 \forall T>0\ \ \ \lin\sup_{t\in[0,T]}|\xi_n(t)-\xi(t)|=0\ \ \mbox{a.s.}
 $$
 Then we have the following almost sure convergence 
 $$
 \forall T>0\ \ \  \frac 1n\sum_{k\leq {nT}}f(\xi_n(\frac kn), \eta_k) \toas  \int_0^T\bar f(\xi(t))dt, \ n\to\infty,
 $$
where $\bar f(x)=\E f(x,\eta_k).$ 
\end{lem}
\begin{proof}
For simplicity let us prove the lemma for $T=1$ only.

Let $\ve>0$ be arbitrary. Choose $\delta>0$ such that
$$
\forall x,y\in\mbR^d, \ |x-y|<\delta, \ \ \forall z\in\mbR \ \ |f(x,z)- f(y,z)|<\ve.
$$
Let $M>0, N\in\mbN$ be such that
$$
\Pb(\forall n\geq N \ \ \sup_{t\in[0,1]}|\xi_n(t)-\xi(t)|<\delta, \ \sup_{t\in[0,1]}|\xi_n(t)|\leq M)>1-\ve.
$$
Set $\Omega_\ve:=\{\forall n\geq N \ \ \sup_{t\in[0,1]}|\xi_n(t)-\xi(t)|<\delta, \ \sup_{t\in[0,1]}|\xi_n(t)|\leq M\}.$

Then for each $\omega\in \Omega_\ve$ and any $m> 1/\delta$
$$
\left|\frac 1n\sum_{k=1}^n f(\xi_n( k/n), \eta_k) - \frac 1n \sum_{j=0}^{m-1}\sum_{j/m\leq k/n<(j+1)/m}  f(\xi_n( j/m), \eta_k)\right|<\ve.
$$
Observe that for each $\omega\in \Omega_\ve$ and any $m> 1/\delta$
$$
   \frac 1n  \left|\sum_{j=0}^{m-1}  \sum_{j/m\leq k/n<(j+1)/m} \left(  f(\xi_n(j/m ), \eta_k)-  \bar f(\xi_n(j/m)) \right)\right|\leq
$$
$$
    \frac 2n   \sum_{j=0}^{m-1} \max_{p\leq M\delta} \left | \sum_{j/m\leq k/n<(j+1)/m} \left( f( [p\delta], \eta_k)
-  \bar f([p\delta])\right)\right| +2\ve.
$$
Since $f$ is bounded, by the ergodic theorem, for any fixed $m$ we have the convergence
$$
 \frac 2n   \sum_{j=0}^{m-1} \max_{p\leq M\delta} \left | \sum_{j/m\leq k/n<(j+1)/m} \left( f( [p\delta], \eta_k)
-  \bar f([p\delta])\right)\right| 
 \toas 0, n\to\infty.
$$
It follows from the previous estimates that for a.a.  $\omega\in\Omega_\ve$ and all $m> 1/\delta$ 
$$
\limsup_{n\to\infty}\left|\frac 1n\sum_{k=1}^n f(\xi_n(k/n), \eta_k) - \int_0^1\bar f(\xi(t))dt
\right|\leq 
$$
$$
\limsup_{n\to\infty}\left|\frac 1n\sum_{k=1}^n f(\xi_n(k/n), \eta_k)  
- \frac 1n \sum_{j=0}^{m-1}\sum_{j/m\leq k/n<(j+1)/m}  f(\xi_n( j/m), \eta_k)\right|+
$$
$$
\limsup_{n\to\infty}\left| \frac 1n \sum_{j=0}^{m-1}\sum_{j/m\leq k/n<(j+1)/m}  f(\xi_n( j/m), \eta_k)-
\ \frac 1n \sum_{j=0}^{m-1}  \sum_{j/m\leq k/n<(j+1)/m}  \bar f(\xi_n(j/m))\right|+
$$
$$
\limsup_{n\to\infty}\left|\ \frac 1n \sum_{j=0}^{m-1}  \sum_{j/m\leq k/n<(j+1)/m}  \bar f(\xi_n(j/m)) 
- \ \frac 1n \sum_{j=0}^{m-1}  \sum_{j/m\leq k/n<(j+1)/m}  \bar f(\xi(j/m))\right|+
$$
$$
\limsup_{n\to\infty}\left|\frac 1n \sum_{j=0}^{m-1}  \sum_{j/m\leq k/n<(j+1)/m}  \bar f(\xi(j/m)) - \int_0^T\bar f(\xi(t))dt \right|\leq
$$
$$
4\ve+\left|\frac 1m\sum_{j=0}^{m-1}    \bar f(\xi(j/m)) - \int_0^1\bar f(\xi(t))dt \right|.
$$
Passing $m\to\infty$ we get for a.a. $\omega\in\Omega_\ve$
$$
\limsup_{n\to\infty}\left|\frac 1n\sum_{k=1}^n f(\xi_n(k/n), \eta_k) - \int_0^1\bar f(\xi(t))dt
\right|\leq  4\ve.
$$

Since $\ve>0$ were  arbitrary, this completes  the proof of   Lemma \ref{lem3}  and hence   Theorem \ref{thm1}.
\end{proof}
\end{proof}
\begin{remk}
Assumption of boundedness and uniform continuity of $\vf$ may be relaxed.
 
We used boundedness of $\vf$ when we apply dominated convergence theorem in Lemma \ref{lem2},  
 and also when we applied Novikov's theorem  to  \eqref{eq_int}, or applying ergodic theorem in Lemma \ref{lem3}.      

Using truncation arguments  
it can be proved that assumption of boundedness of   $\vf$ can be replaced by the linear growth condition with respect to the   second
argument. 
To guarantee that $p^{(n)}_{i,k}$ in \eqref{eq_ERW1} is a probability we have to define it by
$p^{(n)}_{i,k}=(\frac12 (1+\ve^{(n)}_{k,X^{(n)}(k),i, \omega_k}))\vee 0)\wedge 1.$

If $\vf$  depends only on the first three of its coordinates, i.e.,
 $\vf(t,x,l,\omega)=\vf(t,x,l),$ we  used only   the  continuity  of $\vf$, so   the  uniform continuity condition is an extra assumption.

If $\omega_k$ are bounded random variables, the uniform continuity condition can be replaced by only continuity assumption too.
  \end{remk}

\end{document}